\documentclass[reqno,dvipsnames]{amsart}
\usepackage[utf8]{inputenc}

\usepackage{hyperref}
\usepackage{amsmath, amssymb, amsthm}
\usepackage{amscd}
\usepackage{xcolor}
\usepackage[font=footnotesize]{caption}
\usepackage{bbm}
\usepackage{diagbox, hhline, booktabs}
\usepackage[bottom=1in]{geometry}
\usepackage{ulem} 
\usepackage{enumitem}
\usepackage{youngtab}

\newtheorem{theorem}{Theorem}[section] 
\newtheorem{proposition}[theorem]{Proposition}
\newtheorem{lemma}[theorem]{Lemma}

\newtheorem{conjecture}[theorem]{Conjecture}
\newtheorem{definition}[theorem]{Definition}

\newtheorem{question}[theorem]{Question}

\theoremstyle{definition}

\theoremstyle{remark}
\newtheorem{remark}[theorem]{Remark}

\numberwithin{equation}{section}

\usepackage{tikz}
\usetikzlibrary{decorations.markings,decorations.pathreplacing}
\usetikzlibrary{patterns}

\newcommand{\cC}{\mathcal{C}}
\newcommand{\cF}{\mathcal{LR}}

\newcommand{\sa}{\mathsf{A}}
\newcommand{\ssf}{\mathsf{f}}
\newcommand{\sg}{\mathsf{g}}
\newcommand{\N}{\mathbb{N}}

\allowdisplaybreaks 
\title{Some stable plethysms}

\author{Stacey Law}
\address[S. Law]{Department of Pure Mathematics and Mathematical Statistics, University of Cambridge, Cambridge CB3 0WB, UK}
\email{swcl2@cam.ac.uk}

\author{Yuji Okitani}
\address[Y. Okitani]{Department of Mathematics, University of California, Berkeley, CA 94720, USA}
\email{yuji\_okitani@berkeley.edu}

\begin{document}
	
\begin{abstract}
	In this note, we prove some new stability results for plethysm coefficients. As special cases, we verify a conjecture of Wildon, and show the stability of sequences recently predicted by Bessenrodt, Bowman and Paget to be weakly increasing.
\end{abstract}


\maketitle
\baselineskip=15pt

\section{Introduction}

First introduced by Littlewood in 1936, plethysm is a type of product of two symmetric functions which remains widely studied today. It is one of the most important operations on symmetric functions that is still not fully understood combinatorially: indeed, the question of finding a combinatorial description of the plethysm coefficients $a^\nu_{\lambda,\mu}$, the multiplicity with which the Schur function $s_\nu$ appears in the decomposition of the plethysm product $s_\lambda\circ s_\mu$ as a linear combination of Schur functions, is Problem 9 from Stanley's survey of positivity problems in algebraic combinatorics \cite{Sta-OP}.
Although there has been much work done to understand plethysms in special cases, such as computing $a^\nu_{\lambda,(m)}$ for small values of $|\lambda|$ \cite{Thrall, Foulkes, DS}, describing $a^\nu_{\lambda,\mu}$ when the indexing partitions are of `nice' shapes such as hooks or of small Durfee size \cite{LR,BBP}, or determining the maximal and minimal partitions $\nu$ with respect to the dominance ordering such that $s_\nu$ is a constituent of $s_\lambda\circ s_\mu$ \cite{PW}, the general problem is still very much open.

One family of results on plethysms centres on understanding the behaviour of sequences of plethysm coefficients. While it can be very difficult to calculate the explicit values of specific coefficients, sometimes it is more tractable to compare them with one another and show that certain sequences exhibit properties such as stability or monotonicity, as results in \cite{CT, Brion, Col, dBPW} demonstrate.

Our main result is the stability of sequences in a new family of plethysm coefficients, generalising those considered in \cite{BBP,LO} and by Wildon (in private communication). To describe this motivation and introduce the collection of plethysm coefficients that we consider, we fix some notation.
For partitions $\alpha$ and $\beta$, we define $\alpha+\beta:=(\alpha_1+\beta_1,\alpha_2+\beta_2,\dotsc)$, while $\alpha\sqcup\beta$ denotes the partition obtained by concatenating the parts of $\alpha$ and $\beta$ and then reordering if necessary.
Let $m\in\N$ and $l\in\{0,1,\dotsc,m\}$. For each $j\in\N_0$, we define partitions $\alpha^{l,m}\{j\}$ and $\alpha^{l,m}[j]$ by adding an `arm' and/or `legs' of certain lengths to the Young diagram of $\alpha$: this is illustrated in Figure~\ref{fig:add} and defined explicitly in Definition~\ref{def:li}, where the concepts are extended to skew shapes.

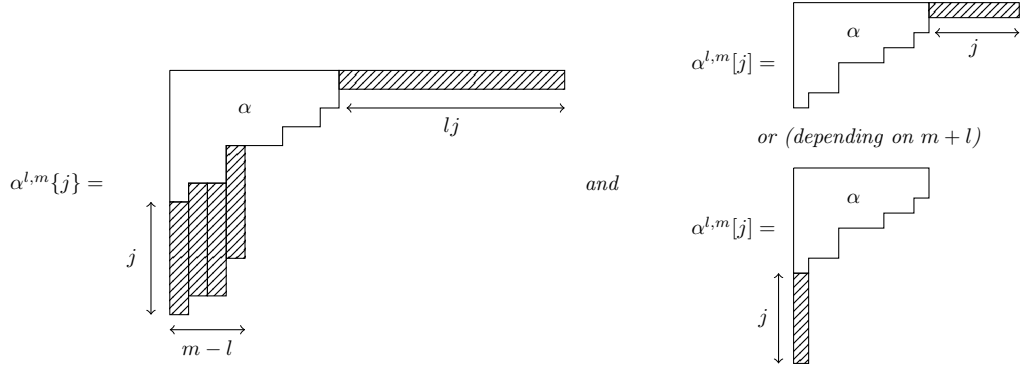
\begin{figure}[h!]
	\centering
	\begin{tikzpicture}[scale=0.5, every node/.style={scale=0.8}]
		\draw (-3,-3) node[] {$\alpha^{l,m}\{j\} =$};
		\draw (0,0) -- (4.5,0) -- (4.5,-1) -- (4,-1) -- (4,-1.5) -- (3,-1.5) -- (3,-2) -- (1.5,-2) -- (1.5,-3) -- (0.5,-3) -- (0.5,-3.5) -- (0,-3.5) -- (0,0);
		\draw (2,-1) node[] {$\alpha$};
		\draw [pattern=north east lines] (0,-3.5) rectangle (0.5,-6.5);
		\draw [pattern=north east lines] (0.5,-3) rectangle (1,-6);
		\draw [pattern=north east lines] (1,-3) rectangle (1.5,-6);
		\draw [pattern=north east lines] (1.5,-2) rectangle (2,-5);
		\draw [pattern=north east lines] (4.5,0) rectangle (10.5,-0.5);
		\draw [<->] (-0.5,-3.5) -- (-0.5,-6.5);
		\draw (-1,-5) node[] {$j$};
		\draw [<->] (4.7,-1) -- (10.5,-1);
		\draw (7.5,-1.5) node[] {$lj$};
		\draw [<->] (0,-6.9) -- (2,-6.9);
		\draw (1,-7.4) node[] {$m-l$};
		\draw (11.5,-3) node[] {\textit{and}};
	\end{tikzpicture}
	\hspace{15pt}
	\begin{tikzpicture}[scale=0.4, every node/.style={scale=0.8}]
		\draw (-2,-2) node[] {$\alpha^{l,m}[j] =$};
		\draw (0,0) -- (4.5,0) -- (4.5,-1) -- (4,-1) -- (4,-1.5) -- (3,-1.5) -- (3,-2) -- (1.5,-2) -- (1.5,-3) -- (0.5,-3) -- (0.5,-3.5) -- (0,-3.5) -- (0,0);
		\draw (2,-1) node[] {$\alpha$};
		\draw [pattern=north east lines] (4.5,0) rectangle (7.5,-0.5);
		\draw [<->] (4.7,-1) -- (7.5,-1);
		\draw (6,-1.5) node[] {$j$};
		\draw (2.5,-4.5) node {\textit{or (depending on $m+l$)}};
		
		\begin{scope}[yshift=-5.5cm]
			\draw (-2,-2) node[] {$\alpha^{l,m}[j] =$};
			\draw (0,0) -- (4.5,0) -- (4.5,-1) -- (4,-1) -- (4,-1.5) -- (3,-1.5) -- (3,-2) -- (1.5,-2) -- (1.5,-3) -- (0.5,-3) -- (0.5,-3.5) -- (0,-3.5) -- (0,0);
			\draw (2,-1) node[] {$\alpha$};
			\draw [pattern=north east lines] (0,-3.5) rectangle (0.5,-6.5);
			\draw [<->] (-0.5,-3.5) -- (-0.5,-6.5);
			\draw (-1,-5) node[] {$j$};
		\end{scope}
	\end{tikzpicture}
	\caption{$\alpha^{l,m}\{j\}$ and $\alpha^{l,m}[j]$ visualised using Young diagrams.} \label{fig:add}
\end{figure}

\begin{definition}\label{def:li}
	Let $m\in\N$ and $l\in\{0,1,\dotsc,m\}$. For any partition $\alpha$ and $j\in\N_0$, we define
	\[ \alpha^{l,m}\{j\}:= \big(\alpha+(lj)\big) \sqcup \big((m-l)^j\big) \quad\text{and}\quad \alpha^{l,m}[j]:= \begin{cases}
		\alpha + (j) & \text{if }m+l\text{ is even},\\
		\alpha\sqcup (1^j) & \text{if }m+l\text{ is odd}.
	\end{cases} \]
	Let $\sigma$ be a skew shape, say $\sigma=\mu/\nu$ for some partitions $\mu$ and $\nu$. Then we define
	\[ \sigma^{l,m}\{j\}:= \big(\mu^{l,m}\{j\}\big)/\nu \quad\text{and}\quad \sigma^{l,m}[j]:=\big(\mu^{l,m}[j]\big)/\nu. \]
	We let $\sa_{l,m}(\sigma,\tau)$ denote the sequence
	\[  \sa_{l,m}(\sigma,\tau):=\big(a^{\sigma^{l,m}\{j\}}_{\tau^{l,m}[j],(m)}\big)_{j\in\N}, \]
	where plethysm coefficients indexed by skew shapes are defined in \eqref{eqn:skew-a} below.
\end{definition}

In their course of their work on classifying multiplicity-free plethysms of Schur functions, Bessenrodt, Bowman and Paget recently conjectured that the sequence of plethysm coefficients $\sa_{1,2}(\nu,\lambda)$ should be weakly increasing, for any partitions $\nu$ and $\lambda$ \cite[Conjecture 1.2]{BBP}.
While we do not prove this monotonicity conjecture, we were able to show that such sequences are always eventually constant, for any $\nu$ and $\lambda$, in \cite[Proposition 5.3]{LO}. Wildon then conjectured a generalisation of this stability result from $m=2$ to arbitrary natural $m$.
\begin{conjecture}[Wildon]\label{conj:mark}
	Let $\nu$ and $\lambda$ be arbitrary partitions, and $m\in\N$. Then $\sa_{m-1,m}(\nu,\lambda)$ stabilises.
\end{conjecture}
We were able to prove Wildon's conjecture, and in fact extend the parameters from $(m-1,m)$ to $(l,m)$ for arbitrary $0\le l\le m$, as well as from partitions $\nu$ and $\lambda$ to arbitrary skew shapes. This is given by the following theorem, which is the main result of this note.

\begin{theorem}\label{thm:main}
	Let $\sigma$ and $\tau$ be any skew shapes.
	For all natural numbers $m$ and for all $l\in\{0,1,\dotsc,m\}$, the sequence of plethysm coefficients $\sa_{l,m}(\sigma,\tau)$ is eventually constant.
\end{theorem}

After introducing the relevant preliminaries in Section~\ref{sec:prelims}, we prove Theorem~\ref{thm:main} in Section~\ref{sec:proof}. We conclude with some observations and open questions in Section~\ref{sec:qs}.

\medskip

\subsection*{Acknowledgements}
We would like to thank Mark Wildon for communicating Conjecture~\ref{conj:mark} to us. We would also like to thank both Mark Wildon and Gunter Malle for their comments on possible further generalisations and open questions, as well as the Mathematisches Forschungsinstitut Oberwolfach (workshop \textit{Character Theory and Categorification}, 2022), where these discussions took place.

\bigskip
\section{Preliminaries}\label{sec:prelims}

We briefly recall the necessary background and notation, and refer the reader to \cite{Mac,Stanley} for further detail. Let $s_\lambda$ denote the Schur function corresponding to the partition $\lambda$, and $\circ$ the plethystic product of symmetric functions. 
For partitions $\lambda$, $\mu$ and $\nu$, the plethysm coefficient $a^\nu_{\lambda,\mu}$ is defined as the multiplicity with which $s_\nu$ appears in the decomposition of the plethysm product $s_\lambda\circ s_\mu$ as a non-negative integer linear combination of Schur functions, namely
\[ a^\nu_{\lambda,\mu} := \langle s_\lambda\circ s_\mu, s_\nu\rangle. \]
Note $a^\nu_{\lambda,\mu}$ equals zero unless $|\nu|=|\lambda|\cdot|\mu|$. It is also straightforward to show that if $\lambda\vdash n$ and $\nu\vdash mn$ then $a^\nu_{\lambda,(m)}=0$ whenever $l(\nu)>n$ (see \cite[Lemma 2.10]{LO}, for example). By convention, we set
\begin{equation}\label{eq:empty}
	a^\nu_{\lambda,\emptyset} = \begin{cases}
		1 & \text{if }\nu=\emptyset\text{ and }l(\lambda)=1,\\
		0 & \text{otherwise}.
	\end{cases}
\end{equation}

For any partition $\lambda=(\lambda_1,\dotsc,\lambda_{l(\lambda)})$ we define its Young diagram $[\lambda]:=\{ (i,j)\in\N\times\N \mid 1\le i\le l(\lambda),\ 1\le j\le \lambda_i\}$. Young diagrams are often depicted using diagrams of boxes as in Figure~\ref{fig:add}, with row and column labellings following that as in a matrix. We also define $\hat{\lambda}$ to be the following:
\[ \hat{\lambda} := \lambda-(1^{l(\lambda)}). \]
In terms of Young diagrams, $[\hat{\lambda}]$ is obtained from $[\lambda]$ by deleting its leftmost column (then shifting what remains one column to the left).
The conjugate of a partition $\lambda$ is denoted by $\lambda'$, and is the partition such that $[\lambda']$ is the transpose of $[\lambda]$. There is a well-known involution on symmetric functions which gives rise to the following (see e.g.~\cite[Ex.~1, Ch.~I.8]{Mac}):

\begin{lemma}\label{lem:inv}
	Let $\lambda$, $\mu$ and $\nu$ be partitions. Then
	\[ a^\nu_{\lambda,\mu} = a^{\nu'}_{\lambda^\ast,\mu'} \quad\text{where}\quad \lambda^\ast:=\begin{cases}
		\lambda & \text{if $|\mu|$ is even},\\
		\lambda' & \text{if $|\mu|$ is odd}.
	\end{cases} \]

\end{lemma}

Next, the so-called Littlewood--Richardson rule allows us to compute the decomposition of the product $s_\lambda s_\mu$ into Schur functions. Specifically, $s_\lambda s_\mu = \sum_\nu c^\nu_{\lambda,\mu} s_\nu$, where the $c^\nu_{\lambda,\mu}$ are called \textit{Littlewood--Richardson coefficients}. One of the various equivalent combinatorial descriptions of the value of $c^\nu_{\lambda,\mu}$ is as follows. 
Let $n\in\N$, $\lambda\vdash n$ and $\cC=(c_1,\dotsc,c_n)$ be a sequence of positive integers. We say that $\cC$ is of type $\lambda$ if $\#\{i\in\{1,\dotsc,n\}\mid c_i=j\}=\lambda_j$ for all $j$. We say an element $c_j$ of $\cC$ is good if $c_j=1$, or if $c_j=i>1$ and $\#\{k\in\{1,\dotsc,j-1\}\mid c_k=i-1\} > \#\{k\in\{1,\dotsc,j-1\}\mid c_k=i\}$. Finally, we say that $\cC$ is good if $c_j$ is good for all $j$. 

\begin{theorem}[Littlewood--Richardson rule]\label{thm:LR}
	Let $\lambda$, $\mu$ and $\nu$ be partitions. Then $c^\nu_{\lambda,\mu}$ is the number of ways of filling each of the boxes of $[\nu]\setminus[\mu]$ with some positive integer such that 
	\begin{enumerate}[label=\textup{(\roman*)}]
		\item reading the numbers from right to left, top to bottom gives a good sequence of type $\lambda$;
		\item the numbers are weakly increasing from left to right along rows; and
		\item the numbers are strictly increasing down columns.
	\end{enumerate}
	In particular, $c^\nu_{\lambda,\mu}=0$ if $|\nu|\ne|\lambda|+|\mu|$.
\end{theorem}

We call such ways of filling $[\nu]\setminus[\mu]$ satisfying Theorem~\ref{thm:LR}(i)--(iii) \textit{Littlewood--Richardson fillings of $\nu/\mu$ of type $\lambda$}, and let $\cF(\nu/\mu)$ denote the set of all Littlewood--Richardson fillings of $\nu/\mu$ of any type. Note for all $\lambda,\mu,\nu$ that $c^\nu_{\lambda,\mu}=c^\nu_{\mu,\lambda}$, and $c^{\nu'}_{\lambda',\mu'}=c^\nu_{\lambda,\mu}$.

We may extend the indexing of Schur functions and plethysm coefficients to skew shapes. A skew shape is given by $\alpha/\beta:=[\alpha]\setminus[\beta]$ where $\alpha$ and $\beta$ are some partitions. (We do not shift $\alpha/\beta$ to require that the top row and leftmost column be non-empty.) The Schur function indexed by the skew shape $\nu/\lambda$ is defined as $s_{\nu/\lambda}:=\sum_\mu c^\nu_{\lambda,\mu} s_\mu$, and accordingly we define
\begin{equation}\label{eqn:skew-a}
	a^{\nu/\lambda}_{\alpha/\beta,\mu}:= \langle s_{\alpha/\beta} \circ s_\mu, s_{\nu/\lambda}\rangle
\end{equation}
for any partitions $\alpha,\beta,\lambda,\mu,\nu$. Moreover, since in this note we only consider the limiting behaviour of sequences of plethysm coefficients of the form $\big(a^{\sigma^{l,m}\{j\}}_{\tau^{l,m}[j],(m)}\big)_j$, when we have a skew shape $\sigma=\alpha/\beta$ we may without loss of generality assume $\beta\subseteq\alpha$.

Finally, we recall the main result of \cite{LO}, which will be useful in the proof of Theorem~\ref{thm:main} below.

\begin{theorem}[{\cite[Theorem A]{LO}}]\label{thm:LO}
	Let $n,m\in\N$, $k\in\{0,1,\dotsc,n-1\}$ and $\lambda\vdash n$. Let $\nu\vdash mn$ with $l(\nu)=n-k$, and $\hat{\nu}=\nu-(1^{l(\nu)})\vdash(m-1)n+k$. Then
	\[ a^\nu_{\lambda',(m)}=\sum_{i=0}^k(-1)^{k+i}\cdot \sum_{\substack{\alpha\vdash k+(m-1)i\\\beta\vdash i}} a^{\alpha/(k-i)}_{\beta',(m)}\cdot a^{\hat{\nu}/\alpha}_{\lambda/\beta,(m-1)}. \]
\end{theorem}

\bigskip
\section{Proof of Theorem~\ref{thm:main}}\label{sec:proof}
We prove Theorem~\ref{thm:main} by induction on $m$, for each fixed $l\in\N_0$. There are three main steps to the proof as follows:
\begin{itemize}
	\item[(a)] For all partitions $\nu,\lambda$, $\sa_{m,m}(\nu,\lambda)$ stabilises for all $m\in\N$, and $\sa_{0,1}(\nu,\lambda)$ is constant (Remark~\ref{rmk:a});
	
	\item[(b)] $\sa_{l,m}(\nu,\lambda)$ stabilises for all partitions $\nu,\lambda \implies \sa_{l,m}(\sigma,\tau)$ stabilises for all skew shapes $\sigma,\tau$;
	
	\item[(c)] $\sa_{l,m}(\sigma,\tau)$ stabilises for all skew shapes $\sigma,\tau \implies \sa_{l,m+1}(\nu,\lambda)$ stabilises for all partitions $\nu,\lambda$.
\end{itemize}

Below, Proposition~\ref{prop:b} gives the proof of Step (b), while Step (c) follows from Theorem~\ref{thm:LO} as shown in the proof of Lemma~\ref{lem:c}.

\begin{remark}\label{rmk:a}
	Step (a) above forms the base case for our induction.
	That $\sa_{m,m}(\nu,\lambda)$ stabilises for all partitions $\nu$ and $\lambda$ follows from \cite{CT} or \cite{Brion}: more generally, for all partitions $\mu$, \cite[Thm 4.2]{CT} states that $\big(a^{\nu+(|\mu|j)}_{\lambda+(j),\mu}\big)_j$ stabilises, in particular with limit zero if $l(\mu)>1$, while \cite[\textsection 2.6, Corollary 1]{Brion} states that $\big( a^{\nu+j\mu}_{\lambda+(j),\mu}\big)_j$ is weakly increasing and eventually constant, where $j\mu:=(j\mu_1,j\mu_2,\dotsc)$.
	
	We could have used the sequence $\sa_{0,0}(\nu,\lambda)$, defined analogously and evaluated using \eqref{eq:empty}, as an alternative starting point, although it is easy enough to observe that $\sa_{l,1}(\nu,\lambda)$ is constant for any $l\in\{0,1\}$ and any partitions $\nu$ and $\lambda$. Indeed, for any $j$ we have that $a^{\nu^{l,1}\{j\}}_{\lambda^{l,1}[j],(1)}=\delta_{\nu^{l,1}\{j\},\lambda^{l,1}[j]}=\delta_{\nu,\lambda}$, where $\delta$ denotes the Kronecker delta.
	\hfill$\lozenge$
\end{remark}

\begin{lemma}\label{lem:c}
	Let $m\in\N$ and $l\in\{0,1,\dotsc,m\}$. Suppose $\sa_{l,m}(\sigma,\tau)$ stabilises for all skew shapes $\sigma$ and $\tau$. Then $\sa_{l,m+1}(\nu,\lambda)$ stabilises for all partitions $\nu$ and $\lambda$.
\end{lemma}

\begin{proof}
	Fix partitions $\nu$ and $\lambda$. For ease of notation, let $\sa_{l,m+1}:=\sa_{l,m+1}(\nu,\lambda)=\big( a^{\nu^{l,m+1}\{j\}}_{\lambda^{l,m+1}[j],(m+1)} \big)_j$, and $\nu^j:=\nu^{l,m+1}\{j\}$ and $\lambda_j:=\lambda^{l,m+1}[j]$. We may without loss of generality assume that $\lambda\vdash n$, $\nu\vdash (m+1)n$ and $l(\nu)=n-k$ for some $n,k\in\N_0$, else $a^{\nu^j}_{\lambda_j,(m+1)}=0$ for all $j$. In particular, $\lambda_j\vdash n+j$ and $\nu^j\vdash (m+1)(n+j)$, with $l(\nu^j)=(n+j)-k$, so by Theorem~\ref{thm:LO} we have for any $j\in\N$ that
	\begin{equation}\label{eqn:c}
		a^{\nu^j}_{\lambda_j,(m+1)} = \sum_{i=0}^k(-1)^{k+i} \sum_{\substack{\alpha\vdash k+mi\\\beta\vdash i}} a^{\alpha/(k-i)}_{\beta',(m+1)}\cdot a^{\hat{\nu^j}/\alpha}_{(\lambda_j)'/\beta,(m)}.
	\end{equation}
	The indexing on the sums in \eqref{eqn:c} is independent of $j$, and so it will suffice to show that each summand $(-1)^{k+i}\cdot a^{\alpha/(k-i)}_{\beta',(m+1)}\cdot a^{\hat{\nu^j}/\alpha}_{(\lambda_j)'/\beta,(m)}$ stabilises as $j \rightarrow \infty$, for any given $i$, $\alpha$, $\beta$. In other words, it suffices to show that the sequence $\big(a^{\hat{\nu^j}/\alpha}_{(\lambda_j)'/\beta,(m)}\big)_j$ stabilises.
	Now $\hat{\nu^j}=\big(\hat{\nu}+(lj)\big)\sqcup\big((m-l)^j\big)$ and
	\[ (\lambda_j)' = \begin{cases}
		\lambda'\sqcup(1^j) & \text{if }m+1+l\text{ is even}\\
		\lambda'+(j) & \text{if }m+1+l\text{ is odd}
	\end{cases}\ =\ \begin{cases}
	\lambda'+(j) & \text{if }m+l\text{ is even}\\
	\lambda'\sqcup(1^j) & \text{if }m+l\text{ is odd},
	\end{cases} \]
	so $\hat{\nu^j}/\alpha = \sigma^{l,m}\{j\}$ and $(\lambda_j)'/\beta=\tau^{l,m}[j]$ where $\sigma=\hat{\nu}/\alpha$ and $\tau=\lambda'/\beta$. By assumption, $\sa_{l,m}(\sigma,\tau)=\big(a^{\sigma^{l,m}\{j\}}_{\tau^{l,m}[j],(m)}\big)_j$ stabilises, and hence $\sa_{l,m+1}=\big(a^{\nu^j}_{\lambda_j,(m+1)}\big)_j$ stabilises by \eqref{eqn:c}.
\end{proof}

\begin{proposition}\label{prop:b}
	Let $m\in\N$ and $l\in\{0,1,\dotsc,m\}$. Suppose $\sa_{l,m}(\nu,\lambda)$ stabilises for all partitions $\nu$ and $\lambda$. Then $\sa_{l,m}(\sigma,\tau)$ stabilises for all skew shapes $\sigma$ and $\tau$.
\end{proposition}

\begin{proof}
	Fix skew shapes $\sigma$ and $\tau$ and suppose $\sigma=\gamma/\alpha$ and $\tau=\delta/\beta$ where $\alpha,\beta,\gamma,\delta$ are partitions. 
	For ease of notation, let $\gamma^j:=\gamma^{l,m}\{j\}$ and $\delta_j:=\delta^{l,m}[j]$, so that $\sigma^{l,m}\{j\}=\gamma^j/\alpha$ and $\tau^{l,m}[j]=\delta_j/\beta$. More generally, for any partitions $\zeta$ and $\eta$ we also let $\zeta^j:=\zeta^{l,m}\{j\}$ and $\eta_j:=\eta^{l,m}[j]$.
	By definition,
	\[ a^{\gamma^j/\alpha}_{\delta_j/\beta,(m)} = \sum_{\substack{\eta\vdash|\delta_j|-|\beta|\\\zeta\vdash|\gamma^j|-|\alpha|}} c^{\delta_j}_{\eta,\beta}\cdot c^{\gamma^j}_{\zeta,\alpha} \cdot a^{\zeta}_{\eta,(m)}. \]
	From Theorem~\ref{thm:LR}, clearly for all sufficiently large $j$ we have a bijection $\cF(\delta_j/\beta)\to\cF(\delta_{j+1}/\beta)$ given by $\ssf\mapsto\sg$, where the Littlewood--Richardson filling $\sg$ is obtained from the filling $\ssf$ by putting the number $x$ into the box $[\delta_{j+1}]\setminus[\delta_j]$: if $\delta_j=\delta+(j)$ then set $x=1$, while if $\delta_j=\delta\sqcup(1^j)$ then set $x=1+\max\{\text{numbers appearing in }\ssf\}$. 
	In particular, if the type of $\ssf$ is $\eta$, then the type of $\sg$ is $\eta+(1)$ or $\eta\sqcup(1)$ respectively, which equals $\eta^{l,m}[1]=\eta_1$ in either case. Hence $c^{\delta_j}_{\eta,\beta}=c^{\delta_{j+1}}_{\eta_1,\beta}$, and moreover, if $c^{\delta_j}_{\epsilon,\beta}>0$ then $\epsilon$ is of the form $\eta_1$ for some $\eta\vdash|\delta_j|-|\beta|$.
	
	Similarly, for all sufficiently large $j$ we have a bijection $\cF\big((\gamma^j)'/\alpha'\big)\to\cF\big((\gamma^{j+1})'/\alpha'\big)$ given by $\ssf\mapsto\sg$, where $\sg$ is obtained from $\ssf$ by putting the number $k$ into the box in row $k$ of $[(\gamma^{j+1})']\setminus[(\gamma^j)']$ for each $k\in\{1,2,\dotsc,m-l\}$, and putting the numbers $w+1,w+2,\dots,w+l$ into the boxes in column 1 of $[(\gamma^{j+1})']\setminus[(\gamma^j)']$, where $w$ is the number filled into the bottom of column 1 of $(\gamma^j)'$ in $\ssf$. In particular, if the type of $\ssf$ is $\zeta'$, then the type of $\sg$ is $\big(\zeta'+(1^{m-l})\big)\sqcup(1^l)=(\zeta^{l,m}\{1\})'=(\zeta^{1})'$.
	
	Now $c^{\gamma^j}_{\zeta,\alpha}=c^{(\gamma^j)'}_{\zeta',\alpha'}$, so there exists $J\in\N$ such that for all $k\ge 0$,
	\[ a^{\gamma^{J+k}/\alpha}_{\delta_{J+k}/\beta,(m)} 
	= \sum_{\substack{\eta\vdash|\delta_J|-|\beta|\\\zeta\vdash|\gamma^J|-|\alpha|}} 
	c^{\delta_{J+k}}_{\eta_k,\beta} \cdot c^{\gamma^{J+k}}_{\zeta^k,\alpha} \cdot a^{\zeta^k}_{\eta_k,(m)} 
	= \sum_{\substack{\eta\vdash|\delta_J|-|\beta|\\\zeta\vdash|\gamma^J|-|\alpha|}} c^{\delta_J}_{\eta,\beta} \cdot c^{\gamma^J}_{\zeta,\alpha} \cdot a^{\zeta^k}_{\eta_k,(m)}. \]
	But $\sa_{l,m}(\zeta,\eta) = \big( a^{\zeta^{l,m}\{k\}}_{\eta^{l,m}[k],(m)} \big)_k = \big( a^{\zeta^k}_{\eta_k,(m)} \big)_k$ stabilises for each $\zeta$ and $\eta$ by assumption, say when $k\ge K(\zeta,\eta)$. Letting $K:=\max\{K(\zeta,\eta)\mid \eta\vdash|\delta_J|-|\beta|, \zeta\vdash|\gamma^J|-|\alpha|\}$, then $a^{\gamma^j/\alpha}_{\delta_j/\beta,(m)} = a^{\gamma^{J+K}/\alpha}_{\delta_{J+K}/\beta,(m)}$ for all $j\ge J+K$. That is, the sequence $\sa_{l,m}(\sigma,\tau) = \big( a^{\sigma^{l,m}\{j\}}_{\tau^{l,m}[j],(m)} \big)_j$ is eventually constant, as desired.
\end{proof}

\begin{proof}[Proof of Theorem~\ref{thm:main}]
	This follows from Remark~\ref{rmk:a}, Lemma~\ref{lem:c} and Proposition~\ref{prop:b}.
\end{proof}

\bigskip
\section{Open questions}\label{sec:qs}

Theorem~\ref{thm:main} gives new families of sequences of plethysm coefficients which are eventually constant. This leads to the following natural questions.

\begin{question}\label{q:malle}
	What is the limit of the sequence $\sa_{l,m}(\sigma,\tau)$? When is this limit strictly positive?
\end{question}

Our original motivation for considering sequences of plethysm coefficients which led us to define $\sa_{l,m}(\sigma,\tau)$ was the conjecture of Bessenrodt, Bowman and Paget that $\sa_{1,2}(\nu,\lambda)$ should be weakly increasing, for all partitions $\nu$ and $\lambda$ \cite[Conjecture 1.2]{BBP}. Furthermore, it follows from \cite[\textsection 2.6, Corollary 1]{Brion} that for all $m\in\N$, the sequence $\sa_{m,m}(\nu,\lambda)$ not only stabilises but is also weakly increasing, for all $\nu$ and $\lambda$. By applying Lemma~\ref{lem:inv}, it immediately follows that $\sa_{0,m}(\nu,\lambda)$ also stabilises and is weakly increasing for all $\nu$ and $\lambda$. Indeed, Brion's result covers coefficients of the form $a^{\nu'+j(1^m)}_{\lambda+(j),(1^m)}$ (see Remark~\ref{rmk:a} above), which by Lemma~\ref{lem:inv} equals $a^{(\nu'+(j^m))'}_{(\lambda+(j))^\ast,(m)} = a^{\nu^{0,m}\{j\}}_{\lambda^{0,m}[j],(m)}$. Thus we make the following generalisation of \cite[Conjecture 1.2]{BBP}.

\begin{conjecture}
	For all $m\in\N$ and $l\in\{0,1,\dotsc,m\}$, and all partitions $\nu$ and $\lambda$, the sequence of plethysm coefficients $\sa_{l,m}(\nu,\lambda)$ is weakly increasing.
\end{conjecture}



Until now, we have considered plethysm coefficients of the form $a^\nu_{\lambda,\mu}$ where $l(\mu)=1$. By applying Lemma~\ref{lem:inv}, we can obtain another family of stable sequences from Theorem~\ref{thm:main} where $(m)$ is replaced by $(1^m)$, and $\sigma^{l,m}\{j\}$ and $\tau^{l,m}[j]$ are conjugated as necessary: for example, it follows immediately that $\big(a^{\nu^{1,2}\{j\}}_{\lambda^{1,2}[j],(1^2)}\big)_j$ also stabilises. Beyond results like this, what other sequences of plethysm coefficients stabilise when we replace $(m)$ by other partitions of $m$?

\smallskip

More generally, it appears that the form of Conjecture~\ref{conj:mark} can be related to certain plethystic semistandard tableaux of maximal weight, as introduced in \cite[Definition 1.4]{dBPW}: we record some observations made by Wildon\footnote{From a research report for the Mathematisches Forschungsinstitut Oberwolfach workshop, \textit{Character Theory and Categorification}, 2022.}. For $j\in\N$ and $\mu$ a partition, a plethystic semistandard tableau of shape $\mu^{(1^j)}$ can equivalently be thought of as a set of $j$ distinct $\mu$-tableaux, which will simply be called a $\mu$-tableau family when $j$ is fixed. A $\mu$-tableau family is maximal if its weight is maximal, with respect to the dominance ordering, over all $\mu$-tableau families. Then a special case of \cite[Theorem 1.5]{dBPW} states that the maximal partitions labelling constituents of $s_{(1^j)}\circ s_\mu$ are precisely the weights of the maximal $\mu$-tableau families.
Now, in the case of $\sa_{1,2}(\nu,\lambda)$ which was considered in \cite[Conjecture 1.2]{BBP} and \cite[Proposition 5.3]{LO}, the way in which the partition $\nu$ grows with $j$ is given by $v^{1,2}\{j\} = \big(\nu+(j)\big)\sqcup(1^j)$. As Wildon notes, $(j,1^j)$ is the weight of the maximal $(1^2)$-tableau family 
\[ {\scriptsize \young(1,2),\ \young(1,3),\ \cdots,\ \young(1,J)} \]
where $J=j+1$. This naturally generalises from $m=2$ to $m\in\N_{\ge 2}$, where we have $\nu^{m-1,m}\{j\}=\big(\nu+((m-1)j)\big)\sqcup(1^j)$, and $\big((m-1)j,1^j\big)$ is the weight of the maximal $(m-1,1)$-tableau family
\[ {\scriptsize \young(1\cdot\cdot\cdot1,2),\  \young(1\cdot\cdot\cdot1,3),\ \cdots,\ \young(1\cdot\cdot\cdot1,J).} \]
\begin{question}[Wildon]
	Can the stability result of \cite[Proposition 5.3]{LO} (or $\sa_{m-1,m}(\nu,\lambda)$ more generally) be generalised further to involve arbitrary maximal tableau families? In other words, given an arbitrary maximal tableau family, is there some stable sequence of plethysm coefficients such that $\nu$ (and possibly $\lambda$) grows with $j$ in a way that depends on the weight of the tableau family?
\end{question}

\bigskip


\begin{thebibliography}{99}
	\bibitem[BBP22]{BBP}
	{\sc C.~Bessenrodt, C.~Bowman and R.~Paget,}
	\newblock The classification of multiplicity-free plethysms of Schur functions,
	\newblock \textit{Trans.~Amer.~Math.~Soc.}, \textbf{375} (2022), 5151--5194.
	
	\bibitem[Bri93]{Brion}
	{\sc M.~Brion,}
	\newblock Stable properties of plethysm: on two conjectures of Foulkes,
	\newblock \textit{Manuscripta Math.} \textbf{80} (4) (1993), 347--371.
	
	\bibitem[Col17]{Col}
	{\sc L.~Colmenarejo,}
	\newblock Stability properties of the plethysm: A combinatorial
	approach,
	\newblock \textit{Discrete Math.} \textbf{340} (2017), 2020--2032.
	
	\bibitem[CT92]{CT}
	{\sc C.~Carr\'e and J-Y.~Thibon,}
	\newblock Plethysm and vertex operators,
	\newblock \textit{Adv.~in Appl.~Math.} \textbf{13} (4) (1992), 390--403.
	
	\bibitem[DS00]{DS}
	{\sc S.~C.~Dent and J.~Siemons,}
	\newblock On a Conjecture of Foulkes,
	\newblock \textit{J.~Algebra} \textbf{226} (2000), 236--249.
	
	\bibitem[dBPW21]{dBPW}
	{\sc M.~de Boeck, R.~Paget and M.~Wildon, }
	\newblock Plethysms of symmetric functions and highest weight representations,
	\newblock \textit{Trans.~Amer.~Math.~Soc.} \textbf{374} (2021), 8013--8043.
	
	\bibitem[Fou54]{Foulkes}
	{\sc H.~O.~Foulkes,}
	\newblock Plethysm of $S$-functions,
	\newblock \textit{Philos.~Trans.~Royal Soc.~A} \textbf{246}	(1954), 555--591.

	\bibitem[LO22]{LO}
	{\sc S.~Law and Y.~Okitani,}
	\newblock On plethysms and Sylow branching coefficients,
	\newblock \textit{Algebr.~Comb.}, to appear.
	
	\bibitem[LR04]{LR}
	{\sc T.~M.~Langley and J.~B.~Remmel,}
	\newblock The plethysm $s_\lambda[s_\mu]$ at hook and near-hook shapes,
	\newblock \textit{Electron.~J.~Comb.} \textbf{11} (2004), \#R11.
	
	\bibitem[Mac95]{Mac}
	{\sc I.~G.~Macdonald,}
	\newblock \textit{Symmetric functions and Hall polynomials}, second ed., 
	\newblock Oxford Classic Texts in the Physical Sciences, The Clarendon Press, Oxford University Press, New York, 1995.

	\bibitem[PW19]{PW}
	{\sc R.~Paget and M.~Wildon,}
	\newblock Generalized Foulkes modules and maximal and minimal constituents of plethysms of Schur functions,
	\newblock \textit{Proc.~London Math.~Soc.} \textbf{118} (2019), 1153--1187.
	
	\bibitem[Sta99]{Stanley}
	{\sc R.~P.~Stanley,}
	\newblock \textit{Enumerative combinatorics}, 
	\newblock Volume 2, Cambridge Studies in Advanced Mathematics \textbf{62}, Cambridge University Press, 1999.
	
	\bibitem[Sta00]{Sta-OP}
	{\sc R.~P.~Stanley,}
	\newblock Positivity problems and conjectures in algebraic combinatorics,
	\newblock in \textit{Mathematics: Frontiers and Perspectives} (V.~Arnold, M.~Atiyah, P.~Lax, and B.~Mazur, eds.), American Mathematical Society, Providence, RI (2000), 295--319; also at \url{https://math.mit.edu/~rstan/papers/problems.pdf}.
	
	\bibitem[Thr42]{Thrall}
	{\sc R.~M.~Thrall,}
	\newblock On symmetrized Kronecker powers and the structure of the free Lie	ring,
	\newblock \textit{Amer.~J.~Math.} \textbf{64} (1942), 371--388.
\end{thebibliography}
\end{document}